\documentclass[14pt]{article}

\usepackage{mathtext}
\usepackage[T1,T2A]{fontenc}
\usepackage[utf8]{inputenc}
\usepackage[english]{babel}

\usepackage{amsfonts, amsmath, amssymb, amsthm, amscd}
\usepackage{mathtools}
\usepackage{ulem}
\usepackage{array}

\usepackage{geometry}
\usepackage{graphicx}
\usepackage{hyperref}
\usepackage{cancel}
\hypersetup{pdfstartview=FitH,  linkcolor=blue,urlcolor=urlcolor, citecolor=blue, colorlinks=true}

\theoremstyle{theorem}
\newtheorem{theorem}{Theorem}[section]
\theoremstyle{theorem}
\newtheorem{proposition}{Proposition}[section]
\theoremstyle{definition}

\theoremstyle{lemma}

\newcommand{\Label}[1]{\label{#1}}

\usepackage{fancyhdr}
\newcommand\Id{1\!\!\mathbb{I}}
\newcommand\RR{\mathbb{R}}
\newcommand\NN{\mathbb{N}}
\newcommand\CC{\mathbb{C}}
\newcommand{\cN}{\mathcal{N}}
\newcommand{\cF}{\mathcal{F}}

\newcommand{\cD}{\mathcal{D}}
\newcommand{\ZZ}{\mathbb{Z}}
\newcommand{\mfM}{\mathfrak{M}}
\newcommand{\al}{\alpha}

\newcommand{\pr}{\prime}
\newcommand{\Qp}{\mathbb Q_p}
\newcommand{\Fx}{\cF_{x\to \xi}}
\newcommand{\Fi}{\cF^{-1}_{\xi\to x}}
\newcommand{\w}[1]{\widetilde{#1}}
\newcommand{\wh}[1]{\widehat{#1}}
\newcommand{\DAN}{\Delta^\al_N}
\newcommand{\vep}{\varepsilon}
\newcommand{\mes}{\text{mes\,}}

\newcommand{\vph}{\varphi}
\newcommand{\si}{\sigma}
\newcommand{\la}{\lambda}
\newcommand{\ga}{\gamma}
\newcommand{\be}{\beta}

\newcommand{\OO}{\Omega}

\newcommand{\Har}{{H^{\al}_{\OO,\rho}}}
\newcommand{\Hao}{{H^{\al}_{\OO,0}}}
\newcommand{\Hag}{{H^{\al}_{\OO,g}}}
\newcommand{\KO}{{(K^n\times K^n)\backslash (\OO^c\times \OO^c)}}
\newcommand{\iKO}{\int\limits_{(K^n\times K^n)\backslash (\OO^c\times \OO^c)}}
\newcommand{\bs}{\backslash}
\newcommand{\bJ}{{\bf J}}

\numberwithin{equation}{section}
\begin{document}

\title{Non-Archimedean Neumann problem: weak and strong solutions}

\author{Alexandra V. Antoniouk\\
Institute of Mathematics, National Academy of Sciences of Ukraine,\\ Tereshchenkivska 3, Kyiv, 01024 Ukraine\\
American University Kyiv, Poshtova Sq 3, 04070, Kyiv, Ukraine\\
E-mail: antoniouk.a@gmail.com
\and
Anatoly N. Kochubei\\
Institute of Mathematics, National Academy of Sciences of Ukraine,\\ Tereshchenkivska 3, Kyiv, 01024 Ukraine \\
E-mail: kochubei@imath.kiev.ua}

\date{}


\maketitle
\bigskip

\begin{abstract}
{
We consider the Neumann problem for the equation with the Vladimirov-Taibleson fractional differentiation operator over a non-Archimedean local field. We study weak solutions following the method by Dipierro, Ros-Oton and Valdinoci (2017). Our investigation of strong solutions is based on the ultrametric identities for the operator under consideration.}

\end{abstract} 

\section{Introduction} \Label{sec1}

Let $K$ be a non-Archimedean local field, $\OO\subset K^n$ be a bounded open set. A model example of a pseudo-differential operator acting on real- or complex-valued functions on $\OO$ is the Vladimirov-Taibleson fractional defferential operator

\begin{equation}\Label{1-1}\
\big(D^{\al,n} u\big)(x)=c_{n,\al} \int\limits_{K^n} \frac{u(x)-u(y)}{\Vert x-y\Vert_{K^n}^{n+\al}}\, dy, \quad x\in \OO,
\end{equation}
where $\Vert z \Vert_{K^n} = \max\limits_{j=1,\ldots,n} \vert z_j\vert_K$ for $z=(z_1,\ldots, z_n) \in K^n$,
\[c_{n,\al}=\frac{q^{\al}-1}{1-q^{-\al-n}}, \quad \al >0,\]
$q$ is the residue cardinality of $K$, and the precise understanding of the integral in \eqref{1-1} depends on smoothness properties of a function $u$.

The operator \eqref{1-1} and related equations have been studied by many authors, see the monographs \cite{VVZ,K2001,AKS,T,KKZ,Z2016} and many recent papers, such as \cite{BGPW,AKN,AKSe,K2023,KK} and many others. This operator is obviously nonlocal; its form resembles that of the fractional Laplacian of real analysis. Both for the latter and $D^{\al,n}$, the Dirichlet problem assign the values of solution not on the boundary of $\OO$, but on its complement, in our case $\OO^c=K^n\backslash \OO$. Note that the most important open subsets of a non-Archimedean local field, like balls and spheres, are simultaneously open and closed, thus have empty boundaries. The classical notion of "boundary value problem" does not make sense in this situation.

Nevertheless a version of the Neumann problem introduced in \cite{DRV} for the fractional Laplacian can be used in our situation. Namely, denote
\begin{equation}\Label{1-2}\
\big(\cN_\al u\big)(x)=c_{n,\al}\int\limits_{\OO} \frac{u(x)-u(y)}{\Vert x-y\Vert_{K^n}^{n+\al}}\, dy, \quad x\in K^n\backslash\OO.
\end{equation}
The non-Archimedean Neumann problem is formulated as follows:
 \begin{equation}\Label{1-3}
	\begin{cases}
		D^{\al,n}u=f &\text{on $\OO$}\\
		\cN_\al\ u =g &\text{on $\OO^c$}
	\end{cases}
\end{equation}
%
where $f,g$ are given functions. Note that a solution $u$ should be defined on the whole space $K^n$ though the first equation in \eqref{1-3} is checked only on $\OO$.

In this paper we develop two alternative settings to understand solution of \eqref{1-3}. While weak solutions are defined and investigated in a way similar to the one developed in \cite{DRV} for real pseudo-differential operators, strong solutions are constructed under the stronger assumptions, using specific ultrametric properties of non-Archiemedean objects.

The structure of this paper is as follows. In Section 2, we present main definitions and facts regarding non-Archimedean local fields. Section 3 is devoted to appropriate analysis of fractional Sobolev spaces, while in Section 4 we deal with some integration by parts formulas. Sections 5 and 6 are devoted to weak solutions. Finally, in Section 7 strong solution are considered.

\setcounter{section}{1} \setcounter{equation}{0} 
\section{Local fields} \Label{sec2}
\setcounter{section}{2} \setcounter{equation}{0}

A (non-Archimedean) local field is a non-discrete totally disconnected locally compact topological field (see for details, e.g. \cite{K2001, Se, We}).

An important construction related to local fields is that of field extension. If a field $k$ is a subfield of field $K$, then $K$ is called an extension of $k$. This relation is often denoted as $K/k$. In the above situation, the field $K$ is a vector space over $k$, and we speak about a finite extension $K/k$, if this vector space is finite dimensional.

A non-Archimedean local field $K$ is isomorphic either to a finite extension of the field $\Qp$ of $p$-adic numbers (where $p$ is a prime number), if $K$ has characteristic zero, or to the field of formal Laurent series with coefficients from a finite field, if $\text{\it char}\, K >0$.

Note that by Ostrowski's theorem, $\Qp$ is the only possible alternative to $\RR$ as a completion of the field of rational numbers. This shows the importance of non-Archimedean mathematics as one of the principal branches of mathematical science.

Any local field is endowed with absolute value $\vert \cdot \vert_K$, such that: 1) $\vert x\vert_K=0$, if and only if $x=0$; 2) $\vert xy\vert_K = \vert x\vert_K\cdot\vert y \vert_K$; 3) $\vert x+y\vert_K\leq \max (\vert x\vert_K, \vert y\vert_K)$. The last property called the ultrametric inequality, implies the equality $\vert x+y\vert_K=\vert x\vert_K$, if $\vert y\vert_K < \vert x\vert_K$.

Note that the ring of integers $O = \{x\in K\colon \vert x\vert_K\leq 1\}$ contains an ideal $P=\{x\in K\colon \vert x\vert_K <1\}$, which in turn contains such an element $\be$ that $P=\be O$. We call the absolute value normalized, if $\vert \be\vert_K = q^{-1}$, where $q$ is the cardinality of the finite field $O/P$. We always assume this property of our absolute value. A normalized absolute value takes exactly the values $q^m$, $m\in\ZZ$.

The additive group of a local field $K$ is self-dual, so that the Fourier analysis in $K$ is similar to the classical one. Let $\chi$ be a fixed non-constant additive character having rank zero, that is $\chi (x)\equiv 1$ as $x$ belongs to the ring of integers, while $\chi(x_0)\neq 1$ for some $x_0\in K$ with $\vert x_0\vert_K = q.$

The Fourier transform of a complex-valued function $f\in L^1(K)$ is defined as
\begin{equation*}
(\cF f)(\xi) \equiv \wh{f}(\xi) = \int\limits_{K} \chi(x\, \xi)f(x)\, dx, \quad \xi\in K,
\end{equation*}
where $dx$ is the Haar measure on the additive group of $K$ normalized in such a way that the measure of $O$ equals 1.
If $\cF f = \wh{f}\in L^1(K)$, then the inversion rule
\begin{equation*}
f(x)  = \int\limits_{K} \chi(-x\, \xi)\wh f(\xi)\, d\xi
\end{equation*}
holds. We will also denote $\w f = \cF^{-1}f$. Sometimes it is convenient to write $\Fx$ and $\Fi$ spe\-ci\-fying the arguments of appropriate functions.
 This Fourier analysis extends easily to functions on $K^n$.

Note that the Haar measure on $K$ or on $K^n$ is unique (up to a multiplicative coefficient) on the Borel $\si$-algebra. Meanwhile it can be extended onto a wider $\si$-algebra $\mfM$ with the following completeness property: if $M\in \mfM$ has the extended measure zero, then every subset of $M$ belongs to $\mfM$. For the proof see Lemma 16.4 in \cite{HR}. Below we assume that the Haar measures have undergone the above extension procedure.

The Fourier transform preserves the Bruhat-Schwartz space $\cD(K)$ of text functions, consisting of locally constant functions with compact supports. The local constancy of a function $f\colon K^n\to \CC$ means the existence of such an integer $\ell$ that for any $s\in K^n$
\[f(x+x^\pr)=f(x), \quad \text{whenever}\ \ \vert x^\pr\vert_K\leq q^{-\ell}. \]
The Fourier transform extends to the dual space $\cD^\pr(K^n)$ called the space of Bruhat-Schwartz distributions.

For $\OO = K^n$ the Vladimirov-Taibleson operator \eqref{1-1} admits, if $u\in\cD(K^n)$, the representation
\[\big(D^{\al,n} u\big)(x)=\Fi \Big(\Vert \xi\Vert^\al_{K^n}\Fx u\Big),\]
where $\Vert \cdot\Vert_{K^n}$ is  the natural non-Archimedean norm on $K^n$:
\[\Vert (x_1,\ldots, x_n)\Vert_{K^n}=\max\limits_{1\leq j \leq n} \vert x_j\vert_K.\]

Below we will need some integration formulas. For methods of calculating some integrals ove subsets of local fields see \cite{T, VVZ,K2001}. We have
\begin{align}\Label{2-1}\
\int\limits_{\Vert x \Vert_{K^n}\leq q^N}\, dx &= q^{nN}, \quad N\in\ZZ;\\
\Label{2-2}
\int\limits_{\Vert y \Vert_{K^n}> q^N}\frac{ dy}{\vert y\vert_K^{\al+1}}&= \big(1-\frac{1}{q}\big)\sum\limits_{j=N+1}^\infty q^{-j(\al+1)}q^j=\frac{1-\frac{1}{q}}{1-q^{-\al}}q^{-(N+1)\al}
=\frac{q-1}{q(q^\al-1)}q^{-N \al};\\
\Label{2-3}
\int\limits_{\Vert x \Vert_{K^n}\leq q^N}\chi(ax)\, dx &
=\begin{cases}
		q^N &\text{if}\quad \vert a\vert_K\leq q^{-N};\\
		0 &\text{if}\quad \vert a\vert_K\geq q^{-N+1}.
	\end{cases}
\end{align}

\section{Fractional Sobolev Space}
\Label{sec3}

From now on, we assume that $\OO$ is a non-Archimedean ball
\[B_N=\{x\in K^n \colon \Vert x\Vert_{K^n}\leq q^N\}.\]

Let us introduce a Hilbert space corresponding to the problem \eqref{1-3}.

Suppose that $\rho \in L^1(K^n\backslash\OO)$; $u,v \colon K^n\to \RR$. Denote
\begin{align}\Label{3-1}\ \nonumber
\Vert u \Vert_\Har &= \Bigg\{\Vert u\Vert^2_{L^2(\OO)} +\Vert \,
\vert \rho\vert^{1/2} u \Vert^2_{L^2(K^n\backslash \OO)}+\\
&+\int\limits_{(K^n\times K^n)\backslash (\OO^c\times \OO^c)} \frac{\vert u(x)-u(y)\vert^2}{\Vert x - y\Vert_{K^n}^{n+\al}}\, dx\, dy\Bigg\}^{1/2};\\ \nonumber
(u,v)_{\Har}&=\int\limits_{\OO} uv \, dx +\int\limits_{K^n\backslash \OO} \vert \rho\vert \, uv\, dx + \\
\Label{3-2}\
&+\iKO \frac{\big(u(x)-u(y)\big)\big(v(x)-v(y)\big)}{\Vert x-y\Vert_{K^n}^{n+\al}}\,dx\,dy.
\end{align}
Correspondingly, the space $\Har=\big\{u\colon K^n\to \RR\  \text{measurable}; \Vert u \Vert_\Har< \infty \big\}$. Our definitions are  similar to that given in \cite{DRV}.
\begin{proposition}\Label{prop1}\
$\Har$ is a real Hilbert space with the inner product \eqref{3-2}.
\end{proposition}

\begin{proof} The nontrivial part of the proof is that of completeness. The scheme of proof follows \cite{DRV}.

Let $u_k\to u$ in $L^2(\OO)$ be a Cauchy sequence with respect to the norm \eqref{3-1}. Taking, if necessary a subsequence we may assume that $u_k\to u$ almost everywhere on $\OO$. Therefore there exists such a subset $Z_1\subset \OO$ that $\text{mes}\,(Z_1)=0$ and $u_k\to u$ for all $x\in \OO \backslash Z_1$.

Denote by $\Id_{(K^n\times K^n)\backslash (\OO^c\times \OO^c)}(x,y)$ the indicator function of the domain of integration in \eqref{3-1}.
Let us write
\[E_u(x,y)=\frac{u(x)-u(y)}{\Vert x-y\Vert_{K^n}^{(n+\al)/2}}\,\Id_{(K^n\times K^n)\backslash (\OO^c\times \OO^c)}(x,y).\]
Then
\[E_{u_k}(x,y)-E_{u_h}(x,y)=\frac{\big(u_k(x)-u_h(x) - u_k(y)+u_h(y)\big)}{\Vert x-y\Vert_{K^n}^{(n+\al)/2}}\,\Id_{(K^n\times K^n)\backslash (\OO^c\times \OO^c)}(x,y).\]

For each $\vep >0$, there exists such a natural number $N_\vep$ that if $k, h \geq N_\vep$, then, since $\{u_k\}$ is the Cauchy sequence,
\[\Vert E_{u_k}-E_{u_h}\Vert^2_{L^2(K^n)}\leq\iKO \frac{\vert (u_k-u_h)(x) - (u_k-u_h)(y)\vert^2}{\Vert x-y\Vert_{K^n}^{n+\al}}\, dx\,dy\leq \vep ^2.\]

Taking a subsequence, if necessary, we find that $E_{u_k}\to E_u$ in $L^2(K^n)$ and almost everywhere on $K^n$. The latter means the existence of such a subset $Z_2 \subset K^n \times K^n$ that $\text{mes}\, (Z_2)=0$ and
\begin{equation}\Label{3-3}\
E_{u_k}(x,y)\to E_u(x,y) \ \text{for all}\  (x.y)\in (K^n\times K^n)\backslash Z_2.
\end{equation}
For $x\in \OO$, let us set $S_x=\{y\in K^n\colon (x,y)\in K^n\backslash Z_2\}$. Denote
\begin{align*}
W &= \{(x,y)\in K^n \times K^n \colon x\in \OO, y \in K^n \backslash S_x\};\\
V&=\{x\in \OO\colon \text{mes}\,(K^n\backslash S_x)=0\}.
\end{align*}
We have
\begin{equation}\Label{3-4}\
W\subseteq Z_2.
\end{equation}

Indeed, if $(x,y)\in W$, then $y\in K^n\backslash S_x$, that is $y\notin S_x$, so that $(x,y)\in Z_2.$

Due to the completeness property of the Haar measure on its extended $\si$-algebra (see Section~\ref{sec2} above), it follows from \eqref{3-4} that the set $W$ is measurable and $\text{mes}\, W=0.$ By Fubini's theorem
\[0=\text{mes}\, W= \int_\OO \text{mes}\, (K^n\backslash S_x)\, dx,\]
so that $\text{mes}\, (K^n\backslash S_x)=0$ for almost all $x\in \OO$. This implies the equality $\text{mes}\, (\OO\backslash V)=0$.
Therefore
\[\text{mes}\, \big(\OO\backslash(V\backslash Z_1)\big)=\text{mes}\,\big((\OO\backslash V)\cup Z_1\big)\leq \text{mes}\, (\OO\backslash V)+ \text{mes}\, Z_1=0,\]
so that $V\backslash Z_1\neq \emptyset$.

Let us fix $x_0\in V\backslash Z_1$. Then $\lim\limits_{k\to \infty}u_k(x_0)=u(x_0).$ In addition, since $x_0\in V$, we have $\text{mes}\,(K^n \backslash S_{x_0})=0.$ This means that for all $y\in S_{x_0}$ (thus for almost all $y\in K^n$) we get $(x_0,y)\in K^n\backslash Z_2$ and by \eqref{3-3},
\[\lim\limits_{k\to\infty}E_{u_k}(x_0,y)=E(x_0,y).\]
Note that $\OO\times \OO^c\subset (K^n\times K^n)\backslash(\OO^c\times\OO^c)$ and
\[E_{u_k}(x_0,y)=\frac{u_k(x_0)-u_k(y)}{\Vert x_0-y\Vert_{K^n}^{(n+\al)/2}}.\]

Therefore for almost every $y\in\OO^c$
\begin{align*}
\lim\limits_{k\to \infty}u_k(y)&=\lim\limits_{k\to \infty}\Big\{u_k(x_0)- \Vert x_0-y\|_{K^n}^{(n+\al)/2}E_{u_k}(x_0,y)\Big\}=\\
&=u(x_0)-\| x_0 - y\|^{(n+\al)/2}E(x_0,y),
\end{align*}
so that in addition to its convergence almost everywhere on $\OO$, the sequence $\{u_k\}$ converges almost everywhere on $K^n\backslash\OO$, thus on $K^n$. Changing the notation, we may assume that $u_k$ converges almost everywhere on $K^n$ to some function $u$. Using Fatou's lemma, for any $\vep > 0$ we find $N_\vep \in \NN$, such that for any $h\geq N_\vep$
\begin{align*}
\vep^2& \geq \liminf\limits_{k\to\infty} \| u_h - u_k\|^2_\Har\geq \liminf\limits_{k\to\infty} \int\limits_\OO (u_h-u_k)^2\, dx+\liminf\limits_{k\to\infty}\int\limits_{\OO^c} |\rho |\, (u_k-u_h)^2 dy\, +\\
&+\liminf\limits_{k\to\infty} \iKO  \frac{\vert (u_k-u_h)(x) - (u_k-u_h)(y)\vert^2}{\Vert x-y\Vert_{K^n}^{n+\al}}\, dx\,dy\geq\\
&\geq \int\limits_\OO (u_h-u)^2 dx + \int\limits_{\OO^c} |\rho |\, (u_h-u)^2 dy\, + \int\limits_{\iKO} \frac{\vert (u_h-u)(x) - (u_h-u)(y)\vert^2}{\Vert x-y\Vert_{K^n}^{n+\al}}\, dx\,dy=\\
&=\| u_h -u \|^2_\Har,
\end{align*}
that is $u_h\to u$ in $\Har$ and this space is complete.
\end{proof}

Remark that even for $\rho = 0$ we should not write $H^\al (\OO)$ instead of $\Har$, since the functions from the latter space are defined on $K^n$.

We will need the following analog of the classical embedding theorem.

\begin{proposition}\Label{prop1}\
Let $F$ be a bounded subset of $\Hao$, $\OO = B_N$, consisting of real-valued functions. If
\begin{equation}\Label{3-5}\
\sup\limits_{f\in F}\int\limits_\OO\int\limits_\OO \frac{|(f(x)-f(y)|^2}{\Vert x-y\Vert_{K^n}^{n+\al}}\,dx\,dy < \infty,
 \end{equation}
 then the set $F$ is precompact in $L^2(\OO)$.
\end{proposition}

\begin{proof} (Compare with \cite{DPV}). We will show that $F$ is totally bounded in $L^2(\OO)$, that is for any $\vep\in (0,1)$ there exist such $\be_1,\ldots, \be_M \in L^2 (\OO)$ that for every function $f\in F$
\begin{equation}\Label{3-6}\
\| f-\be_j\|_{L^2(\OO)}\leq \vep \quad \text{for some} \ j\in\{1,\ldots,M\}.
\end{equation}
Let us consider a finite covering of $\OO$ by non-intersecting small balls $Q_1,\ldots,Q_\ell$ of a radius $\ga = q^{-\nu}$:
\[\OO = \bigsqcup\limits_{j=1}^\ell Q_j, \quad \mes Q_j = q^{-\nu n} = \ga^n \ \text{(see \eqref{2-1}).}
\]
\begin{align}\Label{3-6a}\ \nonumber
&\text{For each}\  x\in\OO\  \text{denote by}\  j(x)\ \text{the unique choice of} \\
&\text{a number from the set}\  \{1,\ldots,\ell\}\ \text{ such that}\  x\in Q_{j(x)}.
\end{align}

 Let
\[ P(f)(x)=\frac{1}{\mes Q_{j(x)}}\int\limits_{Q_{j(x)}} f(y)\, dy.
\]
On each ball $Q_j$, $P(f)(x)$ is constant and will be denoted $r_j(f)$. Note that $P(f+g)=P(f)+P(g)$ for all $f,g \in F$. Denote also
\[R(f) = \ga^{n/2}(r_1(f),\ldots, r_\ell(f)) \in \RR^\ell.
\]
We have $R(f+g) = R(f) +R(g)$,

\[\| P(f)\|^2_{L^2(\OO)} = \sum\limits_{j=1}^\ell\, \int\limits_{Q_j} |P(f)|^2 dx\leq \ga^n \sum\limits_{j=1}^\ell |r_j(f)|^2 = |R(f)|^2 \leq\frac{|R(f)|^2}{\ga^n},
\]
because $\ga < 1$. Next,
\begin{align*}
|R(f)|^2 &=\sum\limits_{j=1}^{\ell} \ga^n |r_j(f)|^2 = \frac{1}{\ga^n} \sum\limits_{j=1}^{\ell} \Big|\int\limits_{Q_j} f(y)\,dy\Big|^2 \leq \sum\limits_{j=1}^{\ell} \int\limits_{Q_j} |f(y)|^2 dy=\\
&=\int\limits_Q |f(y)|^2dy = \| f\|^2_{L^2(\OO)},
\end{align*}
so that the set $R(F)$ is bounded in $\RR^\ell$, thus totally bounded. Therefore for any $\eta >0$ there exist such $b_1,\ldots,b_M\in\RR^\ell$, that
\[R(F)\subseteq\bigcup_{i=1}^M B_\eta (b_i),
\]
where $B_\eta (b)$ is a ball of the radius $\eta$ centered at $b$. For any $i\in\{1,\ldots,M\}$ we write the coordinated of the point $b_i$ as $b_i = (b_{i,1},\ldots,b_{i,\ell})\in\RR^\ell$. For any $x\in\OO$ we set
\[\be_i(x)=\ga^{-n/2}b_{i,j(x)},\]
where $j(x)$ was defined in \eqref{3-6a}. If $x\in Q_j$, then
\[P\big(\be_j(x)\big)=\ga^{-n/2}b_{i,j(x)}=\be_j(x),\]
so that $r_j(\be_i)=\ga^{-n/2} b_{i,j}$, and
\begin{equation}\Label{3-7}\
R(\be_i)=b_i.
\end{equation}
On the other hand
\begin{align*}
\| f- P(f)\|^2_{L^2(\OO)}&=\sum\limits_{j=1}^\ell \,\int\limits_{Q_j} | f(x) - P(f)(x)|^2 dx=\\
&=\sum\limits_{j=1}^\ell \,\int\limits_{Q_j} \Big| f(x) -\frac{1}{\mes Q_j}\int\limits_{Q_j} f(y)\, dy\Big|^2 dx=\\
&=\sum\limits_{j=1}^\ell \,\int\limits_{Q_j} \frac{1}{(\mes Q_j)^2}\, \Big| \int\limits_{Q_j}\big[f(x) -f(y)\big]\, dy\Big|^2 dx\leq\\
&\leq \frac{1}{\ga^{n}}\sum\limits_{j=1}^\ell \,\int\limits_{Q_j} \, \int\limits_{Q_j}|f(x) -f(y)|^2\, dy\, dx.
\end{align*}
Since every point of an ultrametric ball is its center, we find that for $x,y \in Q_j$ we have $\| x-y\|_{K^n}\leq \ga$, so that
\begin{align}\Label{3-8}\ \nonumber
\ga^{-n-\al} &\leq \frac{1}{\| x-y\|_{K^n}^{n+\al}}\ \text{and}\\ \nonumber
\| f - P(f)\|^2_{L^2(\OO)}&\leq \ga^\al  \sum\limits_{j=1}^\ell \,\int\limits_{Q_j} \,\int\limits_{Q_j}\frac{|f(x)-f(y)|^2}{\| x-y\|_{K^n}^{n+\al}}\, dy\,dx\leq\\
&\leq\ga^\al\int\limits_{Q} \,\int\limits_{Q}\frac{|f(x)-f(y)|^2}{\| x-y\|_{K^n}^{n+\al}}\, dy\,dx.
\end{align}
Next,
\begin{equation}\Label{3-9}\
\| f- \be_j\|_{L^2 (\OO)}\leq \| f - P(f)\|_{L^2 (\OO)} + \| P(\be_j) - \be_j\|_{L^2 (\OO)} + \| P(f-\be_j)\|_{L^2. (\OO)}
\end{equation}
It follows from \eqref{3-5} and \eqref{3-8} that choosing $\ga$ sufficiently small we find that the first summand in \eqref{3-9} does not exceed ${\vep}/{2}$ for all $f\in F$. By \eqref{3-7}, the second summand equals zero. The third one equals
\[\frac{|R(f)-R(\be_j)|}{\ga^{n/2}},\]
and we can find such $j\in\{1,\ldots,M\}$ that $R(f)\in B_\eta(b_j)$. Here also $R(\be_j)=b_j.$ Taking $\eta$ sufficiently small, we obtain that the third summand in \eqref{3-9} does not exceed $\vep/2$, which proves \eqref{3-6}.
\end{proof}

\bigskip
See \cite{T,KKZ,G,AKN} for various other approaches to non-Archimedean Sobolev spaces.

\section{Integration by Parts}
\Label{sec4}

Let us give several useful formulas explaining why the operation \eqref{1-2} can be seen as a analogue of the classical normal derivation. For further details see \cite{DRV, FK}.

\begin{itemize}
\item[1)]\ Suppose that $\dfrac{u(x) - u(y)}{\|x-y\|_{K^n}^{n+\al}}\in L^{1}(B_N\times B_N)$, (for example, if $u\in\cD (K^n)$). Then
\begin{equation}\Label{4-1}\
c_{n,\al}\int\limits_{B_N} \int\limits_{B_N}\dfrac{u(x) - u(y)}{\|x-y\|_{K^n}^{n+\al}}\, dx \, dy = c_{n,\al}\int\limits_{B_N} \int\limits_{B_N}\dfrac{u(y) - u(x)}{\|x-y\|_{K^n}^{n+\al}}\, dy \, dx =0
\end{equation}
for the symmetry reasons.
\item[2)]\ Let $u$ be a bounded locally  constant function on $K^n$. Then
\begin{equation}\Label{4-2}\
\int\limits_{\OO} \big(D^{\al,n}u\big) (x)\, dx = - \int\limits_{K^n\backslash \OO} \big(\cN_\al u\big)(y)\, dy.
\end{equation}
This follows from \eqref{4-1}, see \cite{DRV}.
\item[3)]\ Let $u$ and $v$ be bounded locally constant functions on $K^n$. Then
\begin{align}\Label{4-3}\ \nonumber
&c_{n,\al}\!\!\!\!\iKO \frac{\big(u(x)-u(y)\big)\big(v(x)-v(y)\big)}{\Vert x-y\Vert_{K^n}^{n+\al}}\,dx\,dy=\\
&=\int\limits_{\OO} v(x) \big(D^{\al,n}u\big)(x)\, dx + \int\limits_{\OO^c} v(y) \big(\cN_\al u\big)(y)\, dy.
\end{align}
This follows from \eqref{1-1}, \eqref{1-2} and \eqref{4-1}.
\end{itemize}

\section{Weak Solutions}
\Label{sec5}

Let $f\in L^2(\OO)$, $g\in L^1 (K^n\bs \OO)$. A function $u\in \Hao$ is called {\it a weak solution} of the problem \eqref{1-3}, if
\[\frac{c_{n,\al}}{2}\!\!\iKO \frac{\big(u(x)-u(y)\big)\big(v(x)-v(y)\big)}{\Vert x-y\Vert_{K^n}^{n+\al}}\,dx\,dy= \int\limits_{\OO} f(x) v(x)\, dx +\int\limits_{K^n\bs \OO} g(y) v(y)\, dy
\]
for each $v\in \Hag$.

Let us show following \cite{DRV} that the above notion agrees with a variational setting.

\begin{proposition}\Label{prop2}\
Under the above assumptions regarding $f,g$ define a functional ${\bf J}\colon \Hag \to \RR$ by the formula
\[\bJ [u]=\frac{c_{n,\al}}{4}\!\!\iKO \frac{|u(x)-u(y)|^2}{\Vert x-y\Vert_{K^n}^{n+\al}}\,dx\,dy\,- \int\limits_{\OO} f(x) u(x)\, dx\, -\int\limits_{K^n\bs \OO} g(y) u(y)\, dy.
\]
Then every critical point of $\bJ$, that is such a function $u\in\Hag$ that
\[\lim\limits_{\vep\to 0} \frac{\bJ [u+\vep v]- \bJ[u]}{\vep}=0, \quad \text{\rm for all}\ v\in\Hag,\]
is a weak solution of the problem \eqref{1-3}.
\end{proposition}
\begin{proof} It is easy to show that
\[\big| \bJ [u]\big| \leq C \| u\|_{\Hag} < \infty,\]
so that the functional $\bJ$ is well defined on $\Hag$. Next,
\begin{align*}
\bJ[u+\vep v]&= \bJ[u] + \vep \Bigg( \frac{c_{n,\al}}{2}\iKO\frac{\big(u(x)-u(y)\big)\big(v(x)-v(y)\big)}{\Vert x-y\Vert_{K^n}^{n+\al}}\,dx\,dy-\\
&-\int\limits_\OO f(x)v(x)\, dx - \int\limits_{\OO^c}g(y) v(y)\, dy\Bigg) + \frac{c_{n,\al}}{4}\vep^2\iKO \frac{|v(x)-v(y)|^2}{\Vert x-y\Vert_{K^n}^{n+\al}}\,dx\,dy.
\end{align*}
So that
\begin{align*}
&\lim\limits_{\vep\to 0} \frac{\bJ [u+\vep v]- \bJ[u]}{\vep} =\\
&= \frac{c_{n,\al}}{2}\iKO\frac{\big(u(x)-u(y)\big)\big(v(x)-v(y)\big)}{\Vert x-y\Vert_{K^n}^{n+\al}}\,dx\,dy
-\int\limits_\OO f(x)v(x)\, dx - \int\limits_{\OO^c}g(y) v(y)\, dy,
\end{align*}
as desired.
\end{proof}

The next result provide a kind of maximum principle for the problem \eqref{1-3}.
\begin{proposition}\Label{prop3}\
Let $u\in \Hag$ satisfy in weak sense the problem \eqref{1-3} with $f\in L^2(\OO)$, $f\geq 0$, $g\in L^1(\OO)$, $g\geq 0$. Then $u=\text{\rm const}$ everywhere except on a set of measure zero.
\end{proposition}
\begin{proof} The function $v\equiv 1$ belongs to $\Hag$, so that t can be used as a test function in the definition of a weak solution. Therefore
\[0\leq \int\limits_\OO f(x)\, dx = - \int\limits_{K^n \bs\OO}g(y)\,dy\leq 0.\]
This implies the equalities $f(x)=0$ almost everywhere on $\OO$, $g(y)=0$ almost everywhere on $K^n\bs\OO$.

Next, let us take $v=u$ as a test function. We obtain that
\[\iKO\frac{|u(x)-u(y)|^2}{\Vert x-y\Vert_{K^n}^{n+\al}}\,dx\,dy=0,\]
so that $u(x)-u(y)=0$ for $(x,y) \not\in M$, $\mes M = 0$.

For $x\in\OO$, consider the set $N_x=\{y\in\OO\colon (x,y)\in M\}$. It is known that $\mes N_x=0$ for almost all $x$ \cite[\S 36, Thm. A]{Ha}.

For $y\not\in N_x$, we have $u(x)=u(y)$, that is $u(y)= \text{const}$ for $y \not\in N_x$.
\end{proof}

\section{Existence and uniqueness}\Label{sec6}

The following theorem is an analog of the main result from \cite{DRV}. As above, we consider the case where $\OO = B_N$.

\begin{theorem}\Label{T1}\
Suppose that $f\in L^2(\OO)$, $g\in L^1(K^n\bs\OO)$, and there exists such a bounded locally constant function $\psi$ that $\cN_\al\psi = g$ on $K^n\bs \OO$. Then the problem \eqref{1-3} has a weak solution in $\Hao$, if and only if
\begin{equation}\Label{6-1}\
\int\limits_{\OO} f(x)\, dx = - \int\limits_{K^n\bs \OO} g(y)\, dy.
\end{equation}
If \eqref{6-1} is satisfied, then the solution is unique, up to an additive constant.
\end{theorem}
\begin{proof} First we consider the case where $g\equiv 0$, and $f\not\equiv 0$. For a given $h\in L^2(\OO)$, we look for a solution $v\in\Hao$ satisfying the equation
\begin{align}\Label{6-2}\
\int\limits_\OO v(x) \vph(x)\, dx +\iKO\frac{(v(x)-v(y))(\vph(x)-\vph(y))}{\Vert x-y\Vert_{K^n}^{n+\al}}\,dx\,dy=\int\limits_{\OO} h(x) \vph(x)\, dx
\end{align}
for all $\vph\in \Hao$, with the homogeneous Neumann condition
\[\big(\cN_\al u\big)(y)=0\quad \text{for all}\ y\in K^n\bs\OO.\]
Let us define the linear functional $S\colon \Hao \to\RR$
\[S(\vph)=\int\limits_\OO h(x)\vph(x)\,dx, \quad \vph\in\Hao.\]
This functional is continuous, since
\[| S(\vph) |\leq \| h\|_{L^2(\OO)}\| \vph \|_{L^2(\OO)}\leq \| h\|_{L^2(\OO)}\| \vph \|_{\Hao},\]
and by the Riesz theorem, the problem \eqref{6-2} with a given $h$ has a unique solution $v\in \Hao$.
Setting $\vph = v$ in \eqref{6-2} we obtain
\begin{equation}\Label{6-3}\
\| v\|_{\Hag}\leq C\| h\|_{L^2(\OO)},
\end{equation}
since $\OO \times \OO \subset \KO$.

Define the operator $T_0\colon L^2(\OO) \to \Hao$ as $T_0h=v$. Let also $T$ denotes the restriction of $T_0$ onto $\OO$:
\[Th = T_0h\Big|_\OO.\]
Then $T$ is a linear operator on $L^2(\OO).$

Repeating the reasoning from \cite{DRV} and using our Proposition 3.2, we show that $T$ is a compact selfadjoint operator on $L^2(\OO)$.

Let us prove that $\text{Ker}\, (I - T)$ coincides with the set of all functions on $K^n$ constant everywhere except a set of measure zero. Indeed, such a function $c$ satisfies the equation $D^{\al,n} c+c=c$, and the relation $\cN_\al c = 0$. Conversely, let $v\in L^2(\OO)$ belong to $\text{Ker}\, (I - T)$. Consider $T_0v \in\Hao$. By construction,
\begin{equation}\Label{6-3a}\
D^{\al,n}\big(T_0v\big)+T_0v=v\quad \text{on}\ \OO.
\end{equation}
Since also $\cN_\al (T_0 v)=0$ on $K^n\bs\OO$ and for $v\in\text{Ker}\,(I - T)$: $v=Tv=T_0v$ on $\OO$, it follows from \eqref{6-3a} that $D^{\al,n} (T_0v)=0$ on $\OO$. Therefore, Proposition~\ref{prop3} implies that $T_0v$ is a constant on $\OO$, and this is the function $v$.

By the Fredholm alternative, the image $\text{Im} \,(I -T)$ coincides with the orthogonal complement in $L^2(\OO)$ to $\text{Ker}\,(I - T)$. Therefore,
\begin{equation}\Label{6-4}\
\text{Im}\, (I - T) = \{f \in L^2(\OO)\colon \int\limits_{\OO} f(x)\, dx = 0\}.
\end{equation}
Consider $f$ such that $\int\limits_{\OO} f(x)\, dx = 0$. Then, by \eqref{6-4}, there exists $w\in L^2(\OO)$ such that $f=w-Tw$. Define $u\colon=T_0w$. By the definition of $T_0$, $\cN_\al (u)=0$ on $K^n\bs \OO$, and also
\[D^{\al,n}(T_0w)+(T_0w)=w, \quad\text{on}\  \OO.\]
Therefore, on $\OO$,
\[f=w-Tw = w-T_0w = D^{\al,n}(T_0w)=D^{\al,n}(u),\]
and we have obtained the desired solution.

Conversely, if we have a solution $u\in\Hao$ of the problem \eqref{1-3} with $g\equiv 0$, set $w=f+u$. Then we find that
\[D^{\al,n}u+u=f+u=w\quad \text{on}\ \OO.\]
Then $u=T_0w$ on $K^n$, hence $u=Tw$ on $\OO$, so that
\[(I -T)w = w-u = f \quad\text{on}\ \OO\]
and we have that $f\in\text{Im}\, (I - T)$. By \eqref{6-4}, $\int\limits_{\OO}f(x)\,dx=0$, which proves the theorem in  the case $g\equiv 0$.

Let us consider the case, where $g$ is not necessary zero. By our assumptions, there exists a bounded locally constant function $\psi$ such that $\cN_\al \psi = g$ on $K^n\bs \OO$. Let $\bar{u}=u-\psi$. Then $\bar{u}$ is the solution of the Neumann problem

 \begin{equation}\nonumber
	\begin{cases}
		D^{\al,n}\bar{u}=\bar{f} &\text{on $\OO$}\\
		\cN_a\ \bar{u} =0 &\text{on $\OO^c$}
	\end{cases}
\end{equation}
where $\bar{f}=f-D^{\al,n}f$.

As we have proved, this problem has a solution, if and only if $\int\limits_{\OO}\bar{f}\, dx = 0$, that is if and only if
\[0=\int\limits_{\OO}\bar{f}\, dx =\int\limits_{\OO}{f}\, dx -\int\limits_{\OO} \big(D^{\al,n}\psi\big)(x)\, dx.\]
By \eqref{4-2}
\[
\int\limits_{\OO} \big(D^{\al,n}\psi\big) (x)\, dx = - \int\limits_{K^n\backslash \OO} \big(\cN_\al \psi\big)(y)\, dy= - \int\limits_{K^n\backslash \OO} g(y)\, dy.
\]
So that the solvability of our problem is equivalent to the equality \eqref{6-1}.

The uniqueness of a solution is a consequence of Proposition \ref{prop3}.
\end{proof}

\section{Strong Solutions}\Label{sec7}\

Let us consider some cases where the Neumann problem \eqref{1-3} has strong (``classical'') solutions. We begin with the one-dimensional case (n=1) homogeneous problem, that is the case where $g=0$, For brevity, we drop $n=1$ in notations. Thus, our problem now is

	\begin{align}\Label{7-1}\
		D^{\al}u(x)=f(x), &\quad \text{if $| x|_K\leq q^N$}\\
		\Label{7-2}\
		\cN_\al\ u(y) =g(y), &\quad \text{if $| y|_K> q^N$}.
	\end{align}
We are interested in solutions satisfying the problem in a straightforward way, that is belonging to the domain of operators and such that they can be substituted to the equation \eqref{7-1} and the condition \eqref{7-2}.

\vspace{3mm}

{\bf 7.1} Denote by $\DAN$ the regional operator studied (for $n=1$) in \cite{K2018}:

\[\big(\DAN \big)(x)=c_{1,\al} \int\limits_\OO \frac{u(x)-u(y)}{|x-y|^{1+\al}_K}\, dy, \quad x\in\OO = B_N.\]
The above expression makes sense, for example, on locally constant functions. If the operator $D^{\al,n}=D^{\al,1}$ defined in \eqref{1-1} is restricted to functions $u_N$ supported on $B_N$, and the resulting function $D^{\al,1}u_N$ is considered only on $B_N$, we obtain the operator
\[D_N^{\al,1}u_N = \DAN u_N +\la_N u_N,\]
where
\[\la_N = \frac{q-1}{q(1-q^{-\al-1})}q^{-\al N}.\]

On the other hand, by the ultrametric property, for $x\in\OO$
\begin{align*}
\big(D^{\al.1} u\big)(x)&=c_{1,\al}\int\limits_{|y|_K\leq q^N}\frac{u(x)-u(y)}{|x-y|_K^{\al+1}}\, dy +c_{1,\al}\int\limits_{|y|_K> q^N}\frac{u(x)-u(y)}{|x-y|_K^{\al+1}}\, dy=\\
&=\big(\DAN u\big)(x) + c_{1,\al}u(x)\int\limits_{|y|_K> q^N}\frac{1}{|y|_K^{\al+1}}\, dy - c_{1,\al}\int\limits_{|y|_K> q^N}\frac{u(y)}{|y|_K^{\al+1}}\, dy,
\end{align*}
for an arbitrary function $u$, measurable and bounded on $K$ and locally constant on $\OO$.

It follows from \eqref{2-2} that
\begin{equation}\Label{7-3}\
c_{1,\al}\int\limits_{|y|_K> q^N}\frac{1}{|y|_K^{\al+1}}\, dy=\la_N.
\end{equation}

Note that $\la_N$ is the smallest eigenvalue of the operator $D^{\al,1}_N$ in $L^2(\OO)$.

By \eqref{7-3}
\begin{equation}\Label{7-4}\
\big(D^{\al,1} u\big)(x) = \big(\DAN u\big) (x) + \la_N u(x) - c_{1,\al} \int\limits_{|y|_K> q^N}\frac{u(y)}{|y|_K^{\al+1}}\, dy, \quad x\in \OO.
\end{equation}

Similarly, for $x\in\OO^c$, that is for $| x|_K>q^N$,
\[\big(\cN_\al u\big) (x) = c_{1,\al}\int\limits_{|y|_K\leq q^N}\frac{u(x)-u(y)}{|x|_K^{\al+1}}\, dy,
\]
so that
\[\big(\cN_\al u\big) (x) =c_{1,\al}q^N\frac{u(x)}{|x|_K^{\al + 1}} -  c_{1,\al}\frac{1}{|x|_K^{\al+1}}\int\limits_{\OO}u(y)\, dy,
\]
and the condition \eqref{7-2} equivalent to the relation
\begin{equation}\Label{7-5}\
u(y)= q^{-N} \int\limits_{\OO} u(z)\, dz, \quad y\in\OO^c,
\end{equation}
so that a function $u$ satisfying the Neumann condition \eqref{7-2} equals a constant on $\OO^c$ and this constant $h$ coincides with the average \eqref{7-5} of $u$ over $\OO$. Substituting this into \eqref{7-3} we find from \eqref{7-1} that
\begin{equation}\Label{7-6}\
\big(\DAN u\big)(x) + \la_N u(x) - \la_N h = f(x), \quad x\in\OO.
\end{equation}

\medskip
To study the equation \eqref{7-6}, we use the $L^1$-theory of the operator $D^{\al,1}$ developed in \cite{KK, K2018} (see also \cite{AKK}). From now on, we assume that $\al >1$. The operator $A_N$ on $L^1(B_N)$ is defined initially on the set $\cD(B_N)$ of locally constant functions, then extended by duality to the distributions from $\cD^\pr (B_N)$, and then its domain $\cD (A_N)$ in $L^1(B_N)$ is defined on the set of those $\vph \in L^1(B_N)$, for which $A_N\vph=\DAN\vph - \la_N \vph \in L^1(B_N)$. This definition agrees with the ones in terms of the corresponding semigroup of operators or in terms of harmonic analysis on $B_N$. Note that locally constant functions on $B_N$ belong to $Dom(A_N)$.

The resolvent $(A_N+\mu I)^{-1}$, $\mu>0$ was calculated in \cite{K2018}. It has the form
\[\Big( (A_N+\mu I)^{-1}u\Big)(x)=\int\limits_{B_N}r_\mu(x-\xi)u(\xi)\, d\xi +\mu^{-1}q^{-N}\int\limits_{B_N}u(\xi)\, d\xi, \quad u\in L^1 (B_N), \ \mu>0,
\]
where
\[r_\mu(x)=\int\limits_{|\eta|_K\geq q^{-N+1}}\frac{\chi(\eta\, x)}{|\eta|_K^\al-\la_N+\mu}\, d\eta.
\]
We set $\mu = \la_N$ and apply the operator  $(A_N+\la_NI)^{-1}$ to both sides of \eqref{7-6}.
Taking into account our ``boundary condition'' \eqref{7-5} we find that
\begin{equation}\Label{7-7}\
u(x) +\la_N\int\limits_{B_N}r_{\la_N} (x-\xi)u(\xi)\,d\xi=\int\limits_{B_N}r_{\la_N} (x-\xi)f(\xi)\,d\xi, \quad x\in B_N.
\end{equation}
We have proved the following result
\begin{theorem}\Label{T2}\
Suppose that $\al >0$, $f\in L^1(B_N)$, $\int\limits_{B_N} f(x)\,dx=0$.

Define a function $u(x)$ as follows:
\begin{itemize}
\item $u(x)$, $x\in B_N$ is a solution of the Fredholm integral equation \eqref{7-7};
\item $u(y)$ is equal to a constant $h$ satisfying \eqref{7-5} for all $y$, $|y|_K>q^N$.
\end{itemize}
Then $D^\al u\in L^1)B_N)$, and $u$ is a solution of the problem \eqref{7-1}-\eqref{7-2}.
\end{theorem}

Note that already a weak solution of \eqref{7-1}-\eqref{7-2} is unique up to an additive constant. Theorem~ \ref{T2} specifies a family of solutions agreed with the geometry of the problem.

The equation \eqref{7-7} looks like a non-Archimedean analog of the classical Wiener-Hopf equation. However the basic classical tool, the reduction to boundary value problems for analytic functions (see, for example, the recent survey \cite{KAMR}), is not available for the non-Archimedean situation.

\vspace{3mm}

{\bf 7.2} The inhomogeneous Neumann condition
\[\big( \cN_\al u\big) u(x)=g(x), \quad |x|_K >q^N,
\]
with $g\in L^1(K\bs B_N)$, is equivalent to the relation
\[c_{1,\al}\big[q^Nu(x) - \int\limits_{B_N}u(y)\,dy\big]=|x|_K^{1+\al}g(x), \quad |x|_K>q^N.
\]
Repeating the above arguments, we come instead of \eqref{7-6} to the equation
\[
\big(\DAN u\big)(x) + \la_N u(x) = \la_N h +q^{-N}\int\limits_{|z|_K>q^N}g(z)\,dz \cdot |x|_K^{1+\al}+ f(x), \quad x\in K\bs B_N,
\]
which is then used as above.

\vspace{3mm}

{\bf 7.3} The multi-dimensional Neumann problem \eqref{1-1}-\eqref{1-2} can be investigated as follows. Let $L$ be an unramified extension of degree $n$ of the local field $K$. It is known \cite{K2021} that the expansion with respect to the canonical basis in $L$ defines an isometric linear isomorphism between $L$ and $K^n$, which identifies $D^{\al,n}$ defined in \eqref{1-1} with the operator
\begin{equation}\Label{7-8}\
\big(D^\ga_L u\big)(x)=\frac{q^{n\ga}-1}{1-q^{-n(\ga+1)}}\int\limits_L\frac{u(x)-u(z)}{|x-z|^{\ga+1}_L}\, dz,
\end{equation}
where $\ga = \al/n$. The multidimensional Neumann problem can be studied by applying the above one-dimensional results to the operator \eqref{7-8}.

\bigskip
\section {Acknowledgements}
The first author acknowledges the funding support in the framework of the project ``Spectral Optimization: From Mathematics to Physics and Advanced Technology'' (SOMPATY) received from the European Union’s Horizon 2020 research and innovation programme under the Marie Skłodowska-Curie grant agreement No 873071. The second author acknowledges the financial support by the National Research Foundation of Ukraine (Project number 2023.03/0002).


\end{document}